\documentclass{amsart}

\usepackage{amsfonts,amssymb,amscd,amsmath,enumerate,verbatim,calc}

\usepackage{url}
\usepackage[all,cmtip]{xy}
\usepackage{nicefrac}

\newcommand{\lcm}{\operatorname{lcm}}

\newcommand{\BN}{{\mathbb N}}
\newcommand{\bn}{{\boldsymbol n}}

\newcommand{\vv}{\mathcal{V}}
\newcommand{\kk}{\Bbbk}
\newcommand{\Ga}{{\mathbb{G}_a}}

\theoremstyle{plain}
\newtheorem{theorem}{Theorem}

\newtheorem{proposition}[theorem]{Proposition}

\theoremstyle{definition}
\newtheorem{Def}{Definition}

\theoremstyle{remark}

\newcounter{hours}\newcounter{minutes}

\title[  Separating invariants for linear actions of the additive group ]
{ Separating invariants for arbitrary linear actions of the additive group
}

\author[E.~Dufresne]{Emilie Dufresne}
\address{Mathematisches Institut\\
Universit\"at Basel\\
Rheinsprung 21\\
4051 Basel, Switzerland}
\email{emilie.dufresne@unibas.ch}

\author[J.~Elmer]{Jonathan Elmer}
\address{Department of Mathematics\\
University of Aberdeen\\
King's College\\
Aberdeen AB24 3UE \\
Scotland}
\email{j.elmer@abdn.ac.uk}

 \author[M.~Sezer]{M\"uf\.it Sezer}
\address { Department of Mathematics\\ Bilkent University\\ Ankara 06800\\ Turkey}
\email{sezer@fen.bilkent.edu.tr}


\date{\today}

\begin{document}


 \maketitle


\begin{abstract}
We consider   an arbitrary representation of the additive group $\Ga$  over a field of characteristic zero and give
an explicit description of a finite separating set in the corresponding ring of invariants.  \end{abstract}


\section{Introduction}

The problem of distinguishing the orbits of an action of a group $G$ on a vector space $V$ is one of the most fundamental in
mathematics, and some of the most widely studied questions in mathematics are merely special cases of this problem. For
example, if we take $G$ to be the group $GL_n(\kk)$ acting by conjugation on the vector space of $n \times n$ matrices over
a field $\kk$, then this is the problem of classifying square matrices up to conjugacy. If we take $G$ to be the group $SL_2(\kk)$
and $V$ to be the $n$th symmetric power, $S^n(W)$ of the natural representation $W$, then this is the problem of classifying binary
forms of degree $n$ over $\kk$ up to equivalence.

The classical approach to solving these problems is to construct ``invariant polynomials''. These are polynomial functions
$V \rightarrow \kk$ which are constant on the $G$-orbits. One can also view these as the $G$-fixed points $\kk[V]^G$ of the
$\kk$-algebra $\kk[V]$ of polynomial functions from $V$ to $\kk$, where $G$ acts on $\kk[V]$ via
$$g \cdot f (v) = f(g^{-1} \cdot v)$$
for $v \in V$, $g \in G$ and $f \in \kk[V]$. From this point of view it is clear that $\kk[V]^G$ is a subalgebra of $\kk[V]$.
A natural approach to the orbit problem is then to try to find algebra generators.

Invariant theory can be considered to be the study of the subalgebras $\kk[V]^G \subseteq \kk[V]$. The problem of finding
algebra generators has been studied rather extensively over the past 200 years, but we are still a very long way from being
able to write down algebra generators in the general case. For example, in the case of $SL_2(\kk)$ acting on $S^n(W)$, a
complete set of algebra generators is known only for $n \leq 10$, and the number of generators required appears to grow very
quickly with $n$. While the list of groups and representations for which a complete set of generating invariants is known is
very small, the problem has been solved algorithmically for reductive algebraic groups acting on an algebraic variety
(\cite{KemperFiniteGroups, KemperCompRed, DerksenCompRed}) and for certain non-reductive algebraic groups (\cite{VandenEssen}, \cite{DerksenKemperAlgebraic}).
 Many of these algorithms rely on Gr\"{o}bner basis calculations, which have a tendency to explode in higher dimensions. For this
reason, using full sets of generating invariants to separate orbits is rarely a realistic proposition.

It has been known for a number of years that one can sometimes obtain as much information about the orbits of a group using
a smaller subset of $\kk[V]^G$; for a very simple example, see \cite[Example
2.3.9]{DerksenKemper}. With this in mind, a new trend in invariant theory has emerged, based around the following definition:

 \begin{Def}[Derksen and Kemper {\cite[Definition 2.3.8]{DerksenKemper}}]
A \emph{separating set} for the ring of invariants $\kk[V]^G$ is a subset $S \subset \kk[V]^G$ with the following property:
given $v, w \in V$, if there exists an invariant $f$ such that $f(v) \neq f(w)$, then there also exists $s \in S$ such that
$s(v) \neq s(w)$.
\end{Def}

Separating sets have, in many respects, ``nicer'' properties than generating sets. As a first example, it is well known that
if $G$ is finite and the characteristic of $\kk$ does not divide $|G|$, then $\kk[V]^G$ is generated by elements of degree
$\leq |G|$ \cite{FleischmannNoetherBound,FogartyNoetherBound}, but this is not necessarily true in the modular case
\cite{Richman}. On the other hand, the analogue for separating invariants holds in arbitary characteristic
\cite[Theorem~3.9.13]{DerksenKemper}. Second, Nagata famously showed that if $G$ is not reductive, then $\kk[V]^G$ is not
always finitely generated \cite{NagataHilbert14}. On the other hand, regardless of whether $\kk[V]^G$ is finitely generated, it
must contain a finite separating set \cite[Theorem~2.3.15]{DerksenKemper}. Unfortunately, this existence proof is
non-constructive. No algorithm is known for computing finite separating sets of invariants for non-reductive groups.

In this paper, we describe a finite separating set for any finite dimensional representation of the additive group $\Ga$ over
a field $\kk$ of characteristic zero, extending the results of Elmer and Kohls for the indecomposable representations (see
\cite{ElmerKohls}). Accordingly, from now on, $\kk$ denotes a field of characteristic zero and $\Ga$ its additive group. The
group $\Ga$ is in some sense the simplest of all non-reductive groups. We describe briefly its representation theory. In
each dimension there is exactly one indecomposable representation. Following the classical convention, we let $V_{n}$ denote
the indecomposable representation of dimension $n+1$. We have $V_{n}\cong V_{n}^*$. There is a basis $x_0, \dots , x_n$
for $V_n^*$ such that the action of $\Ga$ on $V_n^*$ is given by
\[\alpha \cdot x_i=\sum_{j=0}^i\frac{\alpha^j}{j!}x_{i-j} \text{ for } \alpha \in \Ga, \;\; 0\le i\le n.\]
In this case, we say that $\Ga$ acts \emph{basically} with respect to the basis $\{x_0, \dots , x_n\}$. Note that $\Ga$ acts
on $V_n^*$ via upper triangular and on $V_n$ via lower triangular matrices. We note that $V_n^*$
is isomorphic to the $n$th symmetric power $S^n(V_1^*)$ of $V_1^*$; if $\Ga$ acts basically on
$V_1^*$ with respect to the basis $\{x_0, x_1\}$, then it acts basically on $S^n(V^*_1)$ with respect to the basis
$\{\frac{1}{j!}x_0^{n-j}x_1^j \; 0\le j\le n\}$.

For any finite dimensional representation $W$ of $\Ga$, there is a multiset of non-negative integers
$\bn:=\{n_1,n_2, \ldots ,n_k\}$ such that $W \cong V_{n_1} \oplus V_{n_2} \oplus \ldots \oplus V_{n_k}$ as representations
of $\Ga$. For shorthand, we let $V_{(\bn)}$ denote the latter and identify $\kk[V_{(\bn)}]$
with $\kk[x_{i,j} \mid 0\le i\le n_j, \;  1\le j\le k]$. For convenience, we will assume that $n_1,n_2, \ldots ,n_k$ are
ordered so that $n_j$ is even for $1\le j\le l$ and odd for $l+1\le j\le k$, and further assume that $n_j\equiv 2
\; \mod 4$ for $1\le j\le l'$ and $n_j\equiv 0 \; \mod 4$ for $l'+1\le j\le
l$. As the problem of computing separating sets for indecomposable linear $\Ga$-actions was considered in \cite{ElmerKohls},
we assume throughout that $k \geq 2$.

The main result of this paper is as follows:

\begin{theorem}\label{thm--MainTheoremIntro}
Let $V$ be a finite dimensional representation of $\Ga$, with $\dim(V)=n$. Then there exists a separating set $S \subset \kk[V]^\Ga$ with the
following properties:
\begin{enumerate}
\item $S$ consists of invariants of degree at most $2n-1$.
\item The size of $S$ is quadratic in $n$.
\item $S$ consists of invariants which involve variables coming from at most 2 indecomposable summands. \end{enumerate}
\end{theorem}

This result will be proved in section 2. We also discuss and compare the number and   degrees of elements in $S$ with those
of  generating invariants in known cases. It should be noted that, while we can describe explicitly a separating set for any
ring of invariants of a linear $\Ga$-action, generating sets are known only in small dimensions.  Section
\ref{section-helly} explains the interest in the third property.

This work was carried out during a visit of the first author to Bilkent University funded by T\"uba-Gebip and a later visit
of the second author to Universit\"{a}t Basel. The authors would like to thank Hanspeter Kraft for making this visit
possible.


\section{Separating sets}

Let $V_n$ be the indecomposable representation of $\Ga$ of dimension $n+1$ and suppose $\Ga$ acts basically with respect to
the basis $\{x_0, \dots ,x_n\}$ of $V_n^*$. The action of $\Ga$ is given by the formula
$$\alpha \cdot f=\exp (\alpha D_n) f \text{ for } \alpha \in \Ga, \; f\in \kk[V_n],$$
where $D_n$ is the Weitzenb\"ock derivation

$$D_n=x_0\frac{\partial}{\partial x_1}+\dots +x_{n-1}\frac{\partial}{\partial x_n}.$$

The algebra of invariants $\kk[V_n]^\Ga$ is precisely the kernel of the derivation $D_n$. More generally, the ring of
invariants
 $\kk[V_{(\bn)}]^\Ga$ coincides with the kernel of the
 derivation
$$D_{(\bn)}:=\sum_{j=1}^kx_{0,j}\frac{\partial}{\partial x_{1,j}}+\dots +x_{n_j-1,j}\frac{\partial}{\partial
 x_{n_j,j}}.$$

Let $\bn=(n_1,n_2, \dots ,n_k)$ and $\bn'=(n_1',n_2', \dots ,n_k')$ be two vectors in $\BN^k$ with $n_j\ge n_j'$ for $1\le
j\le k$. Define the linear map $\Pi_{\bn',\bn}:V_{(\bn')} \rightarrow V_{(\bn)}$ to be the map induced by the linear maps $V_{n_j'}\rightarrow
V_{n_j}$,
 \[(a_{0,j},\ldots,a_{n_j',j})\mapsto  (0,\ldots,0, a_{0,j},\ldots,a_{n_j',j}).\]
The map $\Pi_{\bn',\bn}$ is $\Ga$-equivariant, and so we have $\Pi^*_{\bn,\bn'}
(\kk[V_{(\bn)}]^{\Ga})\subseteq
 \kk[V_{(\bn')}]^\Ga$, where $\Pi^*_{\bn,\bn'}$ is the corresponding algebra map. For a vector $\bn=(n_1,n_2, \dots ,n_k)$ set
 $\lfloor \bn/2 \rfloor=(\lfloor  n_1/2 \rfloor,  \lfloor  n_2/2 \rfloor, \dots , \lfloor  n_k/2
 \rfloor)$, where the symbol $\lfloor  x \rfloor$ denotes the
 largest integer less than or equal to $x$. We let $n$ denote the dimension of $V_{(\bn)}$. Note that $n=\sum_{j=1}^k(n_j+1)$.


 \begin{proposition} \label{prop-projection}
 Assume the convention of section 1. Then we have
$$\Pi^*_{\bn,\lfloor \bn/2 \rfloor}(\kk[V_{(\bn)}]^{\Ga})\subseteq \kk[x_{0,j} \mid 1\le j\le l].$$
Moreover, $\Pi^*_{\bn,\lfloor \bn/2 \rfloor}(\kk[V_{(\bn)}]^{\Ga})$ is contained in the ring of invariants of the cyclic group of
order two acting on $\kk[x_{0,j} \mid 1\le j\le l]$ as multiplication by $-1$ on $x_{0,j}$ for $1\leq j\leq l'$ and trivially on the remaining variables.
 \end{proposition}
\begin{proof}
The proof essentially carries over from the indecomposable case (see \cite[Proposition~3.1]{ElmerKohls}). The isomorphisms $V_{n_j}^*\cong S^{n_j}(V_1^*)$ extend the $\Ga$-action on $\kk[V]$ to a $SL_2(\kk)$-action when we identify $V_1^*$ with the natural representation of $SL_2(\kk)$. A well known theorem of
Roberts \cite{Roberts} states that the $\Ga$-equivariant linear map
\[\begin{array}{llcr}
\Phi: & V_{(\bn)} &\longrightarrow & V_{(\bn)}\oplus V_1\\
        & v         & \longmapsto    & (v,(0,1))\end{array}\]
induces an isomorphism $\Phi^*:\kk[V_{(\bn)}\oplus V_1]^{SL_2(\kk)}\rightarrow \kk[V_{(\bn)}]^\Ga$. The elements $\mu_\alpha$ and $\tau$ of
$SL_2(\kk)$  acting on $V_1^*$ via
\[\mu_{\alpha}=\left (\begin{array}{rr}
\alpha
& 0\\
0 & \alpha^{-1}\\
\end{array} \right ) \textrm{ for }\alpha\in \kk\setminus \{0\}, \textrm{ and } \tau=\left (\begin{array}{cc} 0
& -1\\
1 & 0\\
\end{array} \right )\]
act on $V_{(\bn)}\oplus V_1$ as follows:
\[\begin{array}{l}
\mu_{\alpha}\cdot (\ldots, a_{i,j},\ldots,b_0,b_1)=(\ldots, \alpha^{2i-n_j}a_{i,j},\ldots, \alpha^{-1}b_0, \alpha b_1)\\
\tau\cdot (\ldots, a_{i,j},\ldots,b_0,b_1)=(\ldots, (-1)^i\frac{i!}{(n_j-i)!}a_{n_j-i,j},\ldots,b_1,-b_0).\end{array}\]

Let $f\in \kk[V_{(\bn)}]^\Ga$ and pick $h\in \kk[V_{(\bn)}\oplus V_1]^{SL_2(\kk)}$ such that
$\Phi^* (h)=f$. Then $h$ is fixed by $\mu_\alpha$ and so, for all $\alpha\in\kk\setminus\{0\}$,
\[
f(\ldots,a_{i,j},\ldots)=h(\ldots,a_{i,j},\ldots,0,1)=h(\ldots,\alpha^{2i-n_j}a_{i,j},\ldots,0,\alpha).\]
Thus, for all $\alpha\in\kk\setminus\{0\}$, we have
\begin{align*}
(\Pi^*_{\bn,\lfloor \bn/2 \rfloor} f)(\ldots,a_{i,j},\ldots) &=f(\ldots,0,\ldots,a_{0,j},\ldots,a_{\lfloor
n_j/2\rfloor,j},\ldots)\\
&= h(\ldots,0,\ldots,\alpha^{n_j-2\lfloor n_j/2\rfloor}a_{0,j},\ldots,\alpha^{n_j}a_{\lfloor n_j
/2\rfloor,j},\ldots,0,\alpha).\end{align*}

Since this is a polynomial equation in $\alpha$ and $\kk$ is an infinite field, the equality must also hold for $\alpha=0$, in
which case we have:
\[(\Pi^*_{\bn,\lfloor \bn/2 \rfloor} f)(\ldots,a_{i,j},\ldots)=h(\ldots,0,\ldots,a_{0,j},0,\ldots,0,0),\]
where $2|n_j$, proving the first statement.

To prove the second assertion, we use that $h$ is also fixed by $\tau$. We then have
\begin{align*}
(\Pi^*_{\bn,\lfloor \bn/2 \rfloor} f)(\ldots,a_{i,j},\ldots)& =h(\ldots,0,\ldots,a_{0,j},0,\ldots,0,0)\\
                                    &=h(\ldots,0,(-1)^{\nicefrac{n_j}{2}}a_{0,j},0,\ldots,0,0),\end{align*}
ending the proof.
\end{proof}

Let $f,g$ be two polynomials in $\kk[V_{(\bn)}\oplus V_1]^{SL_2(\kk)}$. Assume that
the total degrees of these polynomials in the variables $ y_0,y_1$ are $d_1$ and $d_2$, respectively, where we identify $\kk[V_1]$ with $\kk[y_0,y_1]$. Then for $r\le \min
(d_1,d_2)$,
 the polynomial $$\sum_{q=0}^r(-1)^q{r\choose q}
\frac{\partial^r f}{\partial y_0^{r-q} \partial y_1^{q}} \frac{\partial^r g}{\partial y_0^{q} \partial y_1 ^{r-q}}$$ also
lies in $\kk[V_{(\bn)}\oplus V_1]^{SL_2(\kk)}$ (see, for example, \cite[p. 88]{Olver}).
This polynomial is called the \emph{$r$th transvectant of $f$ and $g$} and is denoted by $\langle f,g \rangle^r$. Together
with Roberts' isomorphism this process produces a new invariant in $\kk[V_{(\bn)}]^\Ga$ from a given pair as follows. Let
$f_1,f_2\in \kk[V_{(\bn)}]^\Ga$, and let $d_1$ and $d_2$ denote the total degrees in $y_0,y_1$ of ${\Phi^*}^{-1}(f_1)$ and
${\Phi^*}^{-1}(f_2)$, respectively. For $r\le \min (d_1,d_2)$ the $r$th \emph{semitransvectant of $f_1$ and $f_2$} is
defined by
$$[f_1,f_2]^r:=\Phi^*(\langle {\Phi^*}^{-1}(f_1), {\Phi^*}^{-1}(f_2)\rangle^r).$$
A crucial part of our separating set consists of semitransvectants of two polynomials each depending on only one summand.
For these invariants, the inverse of Roberts' isomorphism is given in terms of a derivation. For $1\le j\le k$, set
$$\Delta_j=\sum_{i=0}^{n_j}(n_j-i)(i+1)x_{i+1,j}\frac{\partial}{\partial x_{i,j}}.$$
  Let  $f$ be in $\kk[x_{0,j}, x_{1,j}, \dots ,x_{n_j,j}]^{\Ga}$ for some $1\le j\le k$.
Then $f$ is called \emph{isobaric of weight
$m$}, if all of the monomials $x_{0,j}^{e_0}x_{1,j}^{e_1}\cdots x_{n_j,j}^{e_{n_j}}$ in $f$ satisfy $m=\sum_{i=0}^{n_j}(n_j-2i)e_i$.
For an isobaric $f\in \kk[x_{0,j}, x_{1,j}, \dots ,x_{n_j,j}]^{\Ga}$ of weight $m$, the inverse of Roberts' isomorphism is given by
$${\Phi^*}^{-1} (f)=\sum_{i=0}^m(-1)^i\frac{\Delta_j^i(f)}{i!}y_0^iy_1^{m-i},$$
see \cite[p. 43]{HilbertCourse}. For $1\le j_1\neq j_2\le l'$, let $N$ denote the least common multiple of $n_{j_1}$ and
$n_{j_2}$. We define $w_{j_1,j_2}:=[x_{0,j_1}^{N/n_{j_1}}, x_{0,j_2}^{N/n_{j_2}}]^N$.


\begin{proposition}
\label{double} Let $1\le j_1\neq j_2\le l'$. There exists a
non-zero scalar $d$ such that
$$\Pi^*_{\bn,\lfloor \bn/2 \rfloor}(w_{j_1,j_2})=dx_{0,j_1}^{N/n_{j_1}}x_{0,j_2}^{N/n_{j_2}}.$$
\end{proposition}
\begin{proof}
Let $0\le q\le N$ be an integer. Since the weight of the invariant $x_{0,j_1}^{N/n_{j_1}}$ is $N$, the formula for
${\Phi^*}^{-1}$ in the
previous paragraph yields
\begin{alignat*}{2}\frac{\partial^N {\Phi^*}^{-1}(x_{0,j_1}^{N/n_{j_1}})}{\partial
y_0^{N-q} \partial
y_1^{q}}&=\sum_{i=N-q}^{N-q}(-1)^i\frac{\Delta_{j_1}^i(x_{0,j_1}^{N/n_{j_1}})}{i!}
\frac{i!}{(i-N+q)!}\frac{(N-i)!}{(N-i-q)!}y_0^{i-N+q}y_1^{N-i-q}\\
&=(-1)^{N-q}q!\Delta_{j_1}^{N-q}(x_{0,j_1}^{N/n_{j_1}}).
\end{alignat*}
Similarly, we have
\begin{alignat*}{2}\frac{\partial^N {\Phi^*}^{-1}(x_{0,j_2}^{N/n_{j_2}})}{\partial
y_0^{q} \partial
y_1^{N-q}}&=\sum_{i=q}^{q}(-1)^i\frac{\Delta_{j_2}^i(x_{0,j_2}^{N/n_{j_2}})}{i!}
\frac{i!}{(i-q)!}\frac{(N-i)!}{(q-i)!}y_0^{i-q}y_1^{q-i}\\
&=(-1)^{q}(N-q)!\Delta_{j_2}^{q}(x_{0,j_2}^{N/n_{j_2}}).
\end{alignat*}
Using that ${\Phi^*}$ is an algebra homomorphism, we get
$$w_{j_1,j_2}=\sum_{q=0}^N(-1)^qN!\Delta_{j_1}^{N-q}(x_{0,j_1}^{N/n_{j_1}})\Delta_{j_2}^{q}(x_{0,j_2}^{N/n_{j_2}}).$$

Since both $j_1$ and $j_2$ are congruent to two modulo four, we have $\Pi^*_{\bn,\lfloor \bn/2 \rfloor}(x_{i,j})=0$ if
$i<{n_j}/{2}$, and $\Pi^*_{\bn,\lfloor \bn/2 \rfloor}(x_{i,j})=x_{i-\nicefrac{n_j}{2},j}$ if $i-{n_{j}}/{2}\ge 0$ for
$j=j_1,j_2$. Therefore to compute $\Pi^*_{\bn,\lfloor \bn/2 \rfloor}(w_{j_1,j_2})$, it suffices to consider $w_{j_1,j_2}$
modulo the ideal of $\kk[V_{(\bn)}]$ generated by $x_{0,j_1}, \dots , x_{\nicefrac{n_{j_1}}{2}-1,j_1}, x_{0,j_2}, \dots ,
x_{\nicefrac{n_{j_2}}{2}-1,j_2}$. Call this ideal $I$.

A monomial $x_{0,{j_1}}^{e_0}x_{1,{j_1}}^{e_1}\cdots x_{n_{j_1},j_1}^{e_{n_{j_1}}}$ in $\kk[x_{0,j_1}, \dots
,x_{n_{j_1},j_1} ]$ is said to have $j_1$-weight $p$ if $p=\sum_{i=0}^{n_{j_1}} ie_i$. Let $m$ be a monomial with
$j_1$-weight $p$ and $m'$ be any other monomial appearing in $\Delta_{j_1} (m)$. Then $m$ and $m'$ have the same degree and
the $j_1$-weight of $m'$ is $p+1$. It follows that the $j_1$-weight of any monomial appearing in $ \Delta_{j_1}^i
(x_{0,j_1}^{N/n_{j_1}})$ is $i$. But the smallest possible $j_1$-weight of a monomial of degree $N/n_{j_1}$ in
$\kk[x_{\nicefrac{n_{j_1}}{2},j_1}, \dots ,x_{n_{j_1},j_1}]$ is ${N}/{2}$. Hence all monomials of degree $N/n_{j_1}$ of
$j_1$-weight less than ${N}/{2}$ lie in $I$. It follows that $\Delta_{j_1}^i (x_{0,j_1}^{N/n_{j_1}})\in I$ for $i< {N}/{2}$.
Similarly, $\Delta_{j_2}^i (x_{0,j_2}^{N/n_{j_2}})\in I$ for $i < {N}/{2}$. Therefore we have
$$w_{j_1,j_2} \equiv (-1)N!\Delta_{j_1}^{\nicefrac{N}{2}} (x_{0,j_1}^{N/n_{j_1}})\Delta_{j_2}^{\nicefrac{N}{2}}
(x_{0,j_2}^{N/n_{j_2}}) \; \mod I.$$
Furthermore, we claim that $\Delta_{j_1}^{\nicefrac{N}{2}} (x_{0,j_1}^{N/n_{j_1}})$ is equivalent to a non-zero multiple of
$x_{\nicefrac{n_{j_1}}{2},j_1}^{\nicefrac{N}{n_{j_1}}}$ modulo I. To see this, first note that the $j_1$-weight of monomials
appearing in $\Delta_{j_1}^{\nicefrac{N}{2}} (x_{0,j_1}^{N/n_{j_1}})$ is $\frac{N}{2}$. But
$x_{\nicefrac{n_{j_1}}{2},j_1}^{N/n_{j_1}}$ is the only monomial of degree ${N/n_{j_1}}$ in $\kk[x_{\nicefrac{n_{j_1}}{2},
j_1}, \dots ,x_{n_{j_1},j_1}]$ with $j_1$-weight ${N}/{2}$. Thus it suffices to show that
$x_{\nicefrac{n_{j_1}}{2},j_1}^{\nicefrac{N}{n_{j_1}}}$ appears with a non-zero coefficient in
$\Delta_{j_1}^{\nicefrac{N}{2}} (x_{0,j_1}^{\nicefrac{N}{n_{j_1}}})$. This follows because for an arbitrary monomial $m\in
\kk[x_{0,j_1}, \dots ,x_{n_{j_1},j_1} ]$, any monomial that appears in $\Delta_{j_1}(m)$ has positive coefficients, and $x_{\nicefrac{n_{j_1}}{2},j_1}^{\nicefrac{N}{n_{j_1}}}$ can be obtained
from $x_{0,j_1}^{\nicefrac{N}{n_{j_1}}}$ in ${N}/{2}$ steps by replacing a variable $u$ with another variable appearing
in $\Delta_{j_1}(u)$ at each step. This establishes the claim. A similar argument shows that $\Delta_{j_2}^{\nicefrac{N}{2}}
(x_{0,j_2}^{\nicefrac{N}{n_{j_2}}})$ is equivalent to a non-zero multiple of
$x_{\nicefrac{n_{j_2}}{2},j_2}^{\nicefrac{N}{n_{j_2}}}$ modulo I. The assertion of the proposition now follows because
$\Pi^*_{\bn,\lfloor \bn/2 \rfloor}$ is an algebra homomorphism and $\Pi^*_{\bn,\lfloor \bn/2
\rfloor}(x_{\nicefrac{n_{j}}{2},j}^{\nicefrac{N}{n_{j}}})=x_{0,j}^{\nicefrac{N}{n_{j}}}$
for $j=j_1,j_2$.
\end{proof}


We introduce some invariants which will play a key role in the
construction of our separating set, as they did in the construction
of separating sets for the indecomposable representations, see
\cite{ElmerKohls}. For $1\le j \le k$ and $1\le i \le \lfloor
n_{j}/2 \rfloor$ define
$$f_{i,j}:=\sum_{q=0}^{i-1}(-1)^qx_{q,j}x_{2i-q,j}+\frac{1}{2}(-1)^ix_{i,j}^2$$
and $f_{0,j}=x_{0,j}$. Also, for  $1\le i \le \lfloor
\frac{n_{j}-1}{2} \rfloor$ set
$$s_{i,j}:=\sum_{q=0}^{i}(-1)^q\frac{2i+1-2q}{2}x_{q,j}x_{2i+1-q,j}$$
and $s_{0,j}=x_{1,j}$. Note that we have
$D_{(\bn)}(s_{i,j})=f_{i,j}$ for $0\le i \le \lfloor
\frac{n_{j}-1}{2} \rfloor$. An element $f$ in $\kk[V_{(\bn)}]$
is called a local slice if $D_{(\bn)}(f)\in
\kk[V_{(\bn)}]^{\Ga}$. For a non-zero element $f\in
\kk[V_{(\bn)}]$, let $\nu (f)$ denote the maximum integer $d$ such
that $D_{(\bn)}^{d}(f)\neq 0$. For a local slice $s$ and an
arbitrary polynomial $f$ define
$$\epsilon_s(f):=\sum_{q=0}^{\nu (f)}\frac{(-1)^q}{q!}(D_{(\bn)}^qf)s^q(D_{(\bn)} s)^{\nu (f)-q}.$$
We remark that $\epsilon_s(f)\in \kk[V_{(\bn)}]^{\Ga}$. Furthermore, for $l'+1\le j\le l$, we define
$z_j:=[x_{0,j},f_{n_j/4,j}]^{n_j}$. We can now make our main result precise:


\begin{theorem}\label{thm-MainThm}
Let $T$ denote the union of the following set of polynomials in
$\kk[V_{(\bn)}]^\Ga$.
\begin{enumerate}

\item $f_{i,j}$ for $1\le j \le k$ and $0\le i \le \lfloor n_{j}/2
\rfloor$.
 \item $\epsilon_{s_{i_2,j_2}}(x_{i_1,j_1})$ for $1\le j_1<j_2 \le
k$, $\left\lfloor \frac{n_{j_1}-1}{2}
  \right\rfloor< i_1 \le  n_{j_1}$ and   $0\le i_2 \le \left\lfloor
\frac{n_{j_2}-1}{2} \right\rfloor$.

\item $\epsilon_{s_{i_2,j}}(x_{i_1,j})$ for $1\le j\le k$,  $0\le
i_2 \le \left\lfloor \frac{n_{j}-1}{2} \right\rfloor$, $i_2\le i_1\le n_j$.

\item $\epsilon_{s_{i_2,j_2}}(x_{i_1,j_1})$ for $1\le j_2<j_1\le
k$, $0\le i_1\le n_{j_1}$, $0\le i_2 \le \left\lfloor
\frac{n_{j_2}-1}{2} \right\rfloor$.

\item $w_{j_1,j_2}$ for
$1\le j_1\neq j_2\le l'$.

\item $z_j$ for $l'+1\le j\le l$.
\end{enumerate}
Then  $T$ is a separating set for $\kk[V_{(\bn)}]^\Ga$.
\end{theorem}

 \begin{proof}
We first show that the invariants labelled (1)-(4) above separate any pair of vectors that do not
simultaneously lie in $\vv_{V_{(\bn)}}(x_{i_2,j_2} \mid 1\le j_2\le k, \; 0\le i_2 \le \lfloor \frac{n_{j_2}-1}{2}
\rfloor)$. If $v_1=(a_{i,j})$ and $v_2=(b_{i,j})$ are any two such vectors,
then there exists $1\le j'\le k$ such that, for some $0\le i'\le \lfloor \frac{n_{j'}-1}{2} \rfloor$, $a_{i',j'}$ and
$b_{i',j'}$ are not simultaneously zero. We assume that $i'$ and $j'$ are minimal among such indices, that is, that we have
\begin{enumerate}
\item $a_{i,j'}=b_{i,j'}=0$ for $i<i'$. \item
 $a_{i,j}=b_{i,j}=0$
for $j<j'$ and $0\le i\le \lfloor \frac{n_{j}-1}{2} \rfloor$.
\end{enumerate}
If exactly one of $a_{i',j'}$ and $b_{i',j'}$ is zero, then $f_{i',j'}$ separates $v_1$ and $v_2$. Otherwise the value of
any invariant at $v_1$ and $v_2$ is determined by the set $\{f_{i',j'}, \epsilon_{s_{i',j'}}(x_{i_1,j_1}) \mid 0\le i_1\le
n_{j_1}, \; 1\le j_1\le k \}$. Indeed, as $D_{(\bn)}s_{i',j'}=f_{i',j'}$, the ``Slice Theorem" \cite[2.1]{VandenEssen}
implies that
$$\kk[V_{(\bn)}]^{\Ga}_{f_{i',j'}}= \kk[ \epsilon_{s_{i',j'}}(x_{i_1,j_1}) \mid 0\le i_1\le n_{j_1}, \; 1\le j_1\le k]_{f_{i',j'}}.$$
On the other hand, if $i_1<i'$ and $j_1=j'$ or if $0\le i_1\le \lfloor \frac{n_{j_1}-1}{2} \rfloor$ and $j_1<j'$, then $
\epsilon_{s_{i',j'}}(x_{i_1,j_1})$ vanishes
  at $v_1$ and $v_2$ . It follows that the set
  \begin{alignat*}{1}
  f_{i',j'} &\cup \{\epsilon_{s_{i',j'}}(x_{i_1,j_1}) \mid \lfloor \frac{n_{j_1}-1}{2}
  \rfloor < i_1\le n_{j_1}, \; j_1 < j'\} \cup \{\epsilon_{s_{i',j'}}(x_{i_1,j'}) \mid i'\le i_1\le n_{j'}\}\\
& \cup \{\epsilon_{s_{i',j'}}(x_{i_1,j_1}) \mid j'< j_1, \; 0\le
i_1\le n_{j_1}\}
\end{alignat*}
separates $v_1$ and $v_2$ whenever they are separated by some invariant.

It remains to show that $T$ is a separating set  on the zero set
of the ideal
 $I:=(x_{i_2,j_2}\mid 1\le j_2 \le k,~0\le i_2 \le \lfloor \frac{n_{j_2}-1}{2} \rfloor)$.
 Note that $\kk[V_{(\bn)}]/I\cong \Pi^*_{\bn,\lfloor \bn/2 \rfloor}(\kk[V_{(\bn)}])=\kk[V_{(\lfloor \bn/2 \rfloor)}]$.
 Thus, finding a set which separates on $\vv_{V_{(\bn)}}(I)$ is equivalent to finding a subset
 $E\subseteq \kk[V_{(\bn)}]^\Ga$ such that $\Pi^*_{\bn,\lfloor \bn/2 \rfloor}(E)$ separates
the same points of $V_{(\lfloor \bn/2 \rfloor)}$ as $\Pi^*_{\bn,\lfloor \bn/2 \rfloor}(\kk[V_{(\bn)}]^{\Ga})$. By Proposition
\ref{prop-projection}, $\Pi^*_{\bn,\lfloor \bn/2 \rfloor}(\kk[V_{(\bn)}]^{\Ga})\subseteq \kk[x_{0,j} \mid 1\leq j \leq
l]^{C_2}$, where the cyclic group of order two $C_2$ acts as multiplication by $-1$ on the first $l'$ variables and trivially
on the remaining variables.

Consider the subset $B\subseteq \Pi^*_{\bn,\lfloor \bn/2 \rfloor}(T)$ formed by the following:
 \begin{itemize}
 \item $\Pi^*_{\bn,\lfloor \bn/2 \rfloor}(f_{\lfloor \bn_j/2 \rfloor,j})=x_{0,j}^2$, for $1\leq j \leq l$,
\item $\Pi^*_{\bn,\lfloor \bn/2 \rfloor}(w_{j_1,j_2})=dx_{0,j_1}^{N/{n_{j_1}}}x_{0,j_2}^{N/{n_{j_2}}}$ for $1\le j_1\neq
j_2\le l'$,  where $d\neq 0$ and  $N$ is the least common multiple of $n_{j_1}$ and $n_{j_1}$.
 \item $\Pi^*_{\bn,\lfloor \bn/2 \rfloor}(z_j)=x_{0,j}^3$ for $l'+1\le j\le l$, see \cite[Lemma 5.4]{ElmerKohls}.
 \end{itemize}
Showing that $B$ is a separating set for $\kk[x_{0,j} \mid 1\leq j \leq l]^{C_2}$ will end the proof. More precisely, we
show that value of the generators of $\kk[x_{0,j} \mid 1\leq j \leq l]^{C_2}$ is entirely determined by the value of the
elements of $B$. The ring of invariants is given by
$$\kk[x_{0,j} \mid 1\leq j \leq l]^{C_2}=\kk[x_{0,j_1}x_{0,j_2}, x_{0,j} \mid 1\leq j_1 \leq j_2 \leq l',~ l'+1\leq j\leq l]. $$ Suppose $1\le j_1\neq j_2\le l'$. Note that $N/n_{j_1}$ and $N/n_{j_2}$ are odd integers. On points where either
$x_{0,j_1}^2$ or $x_{0,j_2}^2$ is zero, so is $x_{0,j_1}x_{0,j_2}$. Otherwise, we have
\[x_{0,j_1}x_{0,j_2}=\frac{dx_{0,j_1}^{N/{n_{j_1}}}x_{0,j_2}^{N/{n_{j_2}}}}{d(x_{0,j_1}^2)^{1/2(N/{n_{j_1}}-1)}(x_{0,j_2}^2)^{1/2(N/{n_{j_2}}-1)}}.\]
Now suppose $l'+1\le j\le l$. On points where $x_{0,j}^2$ is zero, so is $x_{0,j}$, and otherwise, $x_{0,j}=x_{0,j}^3/x_{0,j}^2$.
Therefore $B\subseteq \Pi^*_{\bn,\lfloor \bn/2 \rfloor}(T)$ is a separating set.
 \end{proof}


Theorem 1 is an easy consequence of Theorem \ref{thm-MainThm}:

\begin{proof}[Proof of Theorem \ref{thm--MainTheoremIntro} ]
 The degree of the each invariant $f_{i,j}$ is two, and the invariants
$z_j$ all have degree three. The degree of $w_{j_1,j_2}$ is
$N/n_{j_1}+N/n_{j_2}$, where $N=\lcm (n_{j_1},n_{j_2})$. Since   $1 \leq j_1,j_2
\leq l'$, we have $N\le (n_{j_1}n_{j_2})/2$ and so the degree of $w_{j_1,j_2}$
is at most $(n_{j_1}+n_{j_2})/2$. Finally, the
degree of $\varepsilon_{s_{i_2,j_2}}(x_{i_1,j_1})$ is
$\deg(s_{i_2,j_2})i_1+1$ which is less than or equal to   $2 n_{j_1}+1$ which is in turn  at most $2n-1$,
since $n_j \leq n-1$ for all $1 \leq j \leq k$. It then
follows that the degree of each invariant in $T$ is at most
$2n-1$, as claimed.

The number of invariants of the form $f_{i,j}$ in our separating
set is $$\sum_{j=1}^k \left \lfloor \frac{n_j}{2} \right \rfloor +1 \leq
\frac{n+k}{2}.$$ Since $1 \leq k \leq n$, this is linear in $n$. Note at this point that for each $j_1,j_2,i_2$ we have $\epsilon_{s_{i_2,j_2}}(x_{0,j_1})=f_{0,j_1}$, so we have already counted these elements.
The number of further invariants in $T$ of the form
$\varepsilon_{s_{i_2,j_2}}(x_{i_1,j_1})$ is
$$\sum_{j_2=1}^k \sum_{j_1=j_2+1}^k \!\!\!\! n_{j_1} \!\! \left \lfloor \frac{n_{j_2}+1}{2} \right \rfloor +
\sum_{j=1}^k \sum_{i_2=0}^{\left \lfloor \frac{n_j-1}{2} \right \rfloor} \!\!\!\! (n_j-i_2+1)  -1
+ \sum_{j_2=1}^k \sum_{j_1=1}^{j_2-1} \!\!\left \lfloor
\frac{n_{j_1}+2}{2} \right \rfloor \left \lfloor \frac{n_{j_2}+1}{2} \right \rfloor.$$
Here the three terms correspond to the invariants labeled
(4),(3), and (2) in our definition of $T$. Using that for any half-integer $x$ we have $x-\nicefrac{1}{2} \leq \left \lfloor x \right \rfloor \leq x$ (which we also used to derive the third term above), the first term is
bounded above by $$\frac{1}{2}\sum_{j_2=1}^k \sum_{j_1=j_2+1}^k
n_{j_1}(n_{j_2}+1)$$
$$\leq \frac{1}{4}\left( \left(\sum_{j_1=1}^k n_{j_1}\right) \left( \sum_{j_2=1}^k n_{j_2} \right) - \sum_{j=1}^k n_j^2 \right)
+ \frac{k-1}{2} \sum_{j=1}^k n_j$$
$$= \frac{1}{4}(n-k)(n+k-2)-\frac{1}{4}\sum_{j=1}^k n_j^2.$$ For the same reason, the second term is bounded above by
$$\sum_{j=1}^k \frac{1}{2}(n_j+1)^2 - \sum_{j=1}^k\frac{1}{2} \frac{(n_j-2)}{2} \frac{n_j}{2} = \frac{3}{8} \sum_{j=1}^k
n_j^2 + \ \text{linear terms}.$$
The third term is bounded above by $$\sum_{j_2=1}^k \sum_{j_1=1}^{j_2-1} \frac{(n_{j_1}+2)(n_{j_2}+1)}{4}$$
$$= \sum_{j_2=1}^k \sum_{j_1=1}^{j_2-1} \frac{(n_{j_1}+1)(n_{j_2}+1)}{4}+  \sum_{j_2=1}^k \sum_{j_1=1}^{j_2-1} \frac{n_{j_2}+1}{4}$$

$$\leq \frac{1}{8}\sum_{j_1=1}^k \sum_{j_2=1}^{k} (n_{j_1}+1)(n_{j_2} + 1) - \frac{1}{8}\sum_{j=1}^k (n_j+1)^2 +\frac{1}{4}\sum_{j_1=1}^k \sum_{j_2=1}^{k}( n_{j_2}+1) -   \frac{1}{4}\sum_{j=1}^k( n_{j_2}+1)$$
$$= \frac{1}{8}n^2-\frac{1}{8}\sum_{j=1}^k n_j^2 + \frac{1}{4}nk +  \text{linear terms}.$$
Moreover, there are $\frac{1}{2}l'(l'-1)$
invariants of the form $w_{j_1,j_2}$, and $l-l'$ of the form $z_j$.
Ignoring linear terms, the size of $T$ is therefore bounded above
by
$$\frac{1}{4}nk+\frac{3}{8}n^2-\frac{1}{4}k^2+\frac{1}{2}{l'}^2$$
which is indeed quadratic in $n$ as claimed, since $l' \leq k$ and
$k \leq n$. Note that when $k=1$ we get a separating set of size
approximately $\frac{3}{8}n^2$, which coincides with the size of
the separating set found in \cite{ElmerKohls}. Indeed, our separating set specializes to the separating
set found in \cite{ElmerKohls} when $k=1$.

The following tables show the exact size of $T$ for certain
representations $V$ of $\Ga$. It also shows the size of a minimal
generating set $c_{\bn}$ of $\kk[V_{(\bn)}]^{\Ga}$, when this is known. The data for the numbers $c_{\bn}$ was taken from
Andries Brouwer's website \cite{BrouwerWeb}. Note that $nV_k$ is taken to mean the direct sum of $n$ copies of
$V_k$. The generators of $nV_1$ which coincide with our separating set $T$ were first conjectured by Nowicki \cite{NowickiConj}, and first proved by Khoury \cite{KhouryNowickiConj}. The case $nV_2$ was recently solved by Wehlau  \cite{WehlauWeitzenboeck}.

\begin{table}[h]\label{gentable1}
{
\begin{tabular}{c||c|c|c|c|c|c|c|c|c|c}
$V$&$2V_2$ &$3V_2$&$4V_2$&$nV_2$&$V_3$&$2V_3$&$3V_3$&$4V_3$&$5V_3$&$nV_3$ \\
\hline \hline
$|T|$& 10& 21& 36& $2n^2+n$&7 & 24& 51& 108&135 & $5n^2+2n$ \\
\hline
$|c_{\bn}|$&6& 13 & 24& $\frac{1}{6}n(n^2+3n+8)$ &4 &26 &97 &280&689&?\\
\end{tabular}
}
\end{table}

\begin{table}[h]\label{gentable2}
{
\begin{tabular}{c||c|c|c|c|c|c|c|c|c|c}
$V$&$\!V_4\!$ &\!$2V_4$\!&$\!3V_4\!$&$\!4V_4$\!&\!$5V_4$\!&\!$nV_4$\!&\!$V_5$\!&\!$2V_5$\!&\!$nV_5$\!&\!$nV_6$ \\
\hline \hline
$|T|$& 11& 35& 75& 128&195 & $7n^2+4n$& 16& 56& $12n^2+4n$&$\frac{1}{2}(31n^2+9n)$   \\
\hline
$|c_{\bn}|$&5& 28 & 103& 305&? &? &23 &?&?&?\\
\end{tabular}
}
\end{table}

\begin{table}[h]\label{gentable3}
{
\begin{tabular}{c||c|c|c|c|c|c|c|c|c}
$V$&$\!V_1\!\oplus\! V_2\! $ &$\!V_1 \!\oplus\! V_3\!$&$\!V_1 \!\oplus\! V_4\!$&$\!V_1 \!\oplus\! V_5\!$&$\!V_2 \!\oplus\! V_3\!$&$\!V_2 \!\oplus V_4\!$&$\!V_2 \oplus
V_5\!$&$\!V_3 \!\oplus\! V_4\!$&$\!V_3 \!\oplus\! V_5\!$\\ 
\hline \hline
$|T|$& 7& 12& 17& 23&15 &21 & 29&30 & 39\\ 
\hline
$|c_{\bn}|$&5& 13 & 20& 94&15 &18 &92 &63&?\\ 
\end{tabular}
}
\end{table}

To prove that $T$ contains only invariants depending on at most two summands, simply observe that the invariants
$f_{i,j}$ and $z_j$ are non-zero only on the summand $V_{n_j}$ of
$V_{(\bn)}$, while $\varepsilon_{s_{i_2,j_2}}(x_{i_1,j_1})$ and
$w_{j_1,j_2}$ are non-zero on only on $V_{n_{j_1}}$ and $V_{n_{j_2}}$.

\end{proof}

\section{A note on Helly dimension}
\label{section-helly}

In \cite{DomokosSzabo}, the authors define the \emph{Helly
dimension} of an algebraic group as follows:

\begin{Def}[{see \cite[Definition~1.1]{DomokosSzabo}}]
The \emph{Helly dimension} $\kappa(G)$ of an algebraic group is the minimal natural number $d$ such that any finite system
of closed cosets in $G$ with empty intersection, has a subsystem consisting of at most $d$ cosets with empty intersection.
We define $\kappa(G):=\infty$, if there are no such natural numbers.
\end{Def}

They go on to show that if $\kk$ is a field of characteristic zero, and $G$ acts on the affine $\kk$-variety $X:=
\Pi_{i=1}^k X_i$, then there exists a dense $G$-stable open subset $U$ of $X$ and a set $S \subset \kk[X]^G$ of invariants
each depending on at most $\kappa(G)$ indecomposable factors of $X$ such that $S$ is a separating set on $U$
\cite[Theorem~4.1]{DomokosSzabo}. It is easy to see that the Helly dimension of $\mathbb{G}_a$ is two: in characteristic
zero, the additive group does not have any proper nontrivial closed subgroups. That is, its only proper subgroup is $\{0\}$,
and the only possible cosets are singletons. In particular, it follows from their work that for any product of
$\Ga$-varieties, we should be able to find an ideal $I$ of $\kk[X]^{\Ga}$ and a set $S \subset \kk[X]^{\Ga}$ of invariants
each depending on at most two factors, such that $S$ is a separating set on the open set $X \setminus \mathcal{V}(I)$. We
recover this result for representations of $\Ga$ in the first part of the proof of Theorem~\ref{thm-MainThm}. In fact, one
could easily prove the same result directly for a product of arbitrary $\Ga$-varieties by applying the ``Slice Theorem'' with
a local slice depending on just one factor.

For $X$ a product of $G$-varieties, Domokos and Szabo also consider the quantities
\[\sigma(G,X):= \min\{d \mid  \exists S \subset \kk[X]^G, \textrm{ a separating set depending on $d$ factors of $X$}\},\]
and $\delta(G,X)$, defined as the minimum natural number $d$ such that given $x \in X$ with $Gx$ closed in $X$, there exists
a set $\{j_1,j_2,\ldots,j_d\}$ such that the projection $y$ of $x$ onto the subvariety $Y:=\Pi_{i=1}^d X_{j_i}$ has $Gy$
closed in $Y$ with the same dimension as $Gx$. The supremum of these quantities over all possible product varieties are
denoted by $\sigma(G)$ and $\delta(G)$, respectively. They remark that for any unipotent group $G$, $\delta(G) \leq \dim(G)$
\cite[Section~5]{DomokosSzabo}, and in particular $\delta(\Ga)=1$. Finally, they show that for any reductive group $G$, we
have \cite[Lemma~5.9]{DomokosSzabo} \[\sigma(G) \leq \kappa(G)+\delta(G).\]

We do not know whether this  inequality holds for non-reductive groups. If it did, it would follow that, given any
affine $\Ga$-variety $X$, we could find a separating subset of $\kk[X]^G$ depending on at most 3 indecomposable factors of
$X$. Theorem \ref{thm-MainThm}(3) shows that, provided $\Ga$ acts linearly, two factors suffices. It would be interesting
to know whether this holds for products of arbitrary affine $\Ga$-varieties.

\bibliographystyle{plain}
\bibliography{MyBib}

\end{document}